\documentclass[12pt]{article}
\usepackage{amssymb, amsmath, amsthm}
\usepackage[all]{xy}
\usepackage{hyperref}

\newcommand{\mr}[1]{\mathrm{#1}}
\newcommand{\mf}[1]{\mathfrak{#1}}
\newcommand{\mc}[1]{\mathcal{#1}}

\newcommand{\Z}{{\bf Z}}

\newcommand{\zp}{{\bf Z}_p}
\newcommand{\qp}{{\bf Q}_p}

\newcommand{\Cl}{\mr{Cl}}
\newcommand{\mup}{\mu_{p^{\infty}}}
\newcommand{\La}{\Lambda}

\newcommand{\dirlim}[1]{\lim_{\substack{\rightarrow \\ #1}}}

\DeclareMathOperator{\Hom}{Hom}
\DeclareMathOperator{\Gal}{Gal}
\DeclareMathOperator{\Res}{Res}

\newtheorem{theorem}{Theorem}[section]
\newtheorem{proposition}[theorem]{Proposition}
\newtheorem{lemma}[theorem]{Lemma}

\newtheorem{conjecture}[theorem]{Conjecture}

\theoremstyle{definition}
\newtheorem{definition}[theorem]{Definition}

\theoremstyle{remark}

\newtheorem*{ack}{Acknowledgments}

\renewcommand{\baselinestretch}{1.2}
\numberwithin{equation}{section}

\begin{document}

\title{The reciprocity conjecture\\of Khare and Wintenberger}
\author{Romyar T. Sharifi\footnote{Department of Mathematics, University of Arizona, P.O.\ Box 210089, Tucson, AZ 85721, USA, sharifi@math.arizona.edu}}
\date{}
\maketitle

\begin{abstract}
	We prove a strengthening of the ``reciprocity conjecture'' of Khare and Wintenberger.
	The input to the original conjecture is 
	an odd prime $p$, a CM number field $F$ containing 
	the $p$th roots of unity, and a pair $(\mf{q}_1,\mf{q}_2)$ 
	of primes of the maximal totally real subfield $F^+$ of $F$
	that are inert in the cyclotomic $\zp$-extension $F_{\infty}^+/F^+$. 
	In analogy to a statement about generalized Jacobians of curves, the conjecture 
	asserts the equality of two procyclic subgroups of the Galois group of the maximal 
	pro-$p$ extension $\mc{M}$ of $F_{\infty}^+$ that is unramified outside $p$ and abelian over 
	$F^+$.  The first is the intersection with $\Gal(\mc{M}/F_{\infty}^+)$ of the closed subgroup of 
	$\Gal(\mc{M}/F^+)$ generated by the Frobenius elements of $\mf{q}_1$ and $\mf{q}_2$.
	The second is generated by the class of an exact sequence defining 
	the minus part of the $p$-part of the ray class group of $F_{\infty}$ of conductor 
	$\mf{q}_1\mf{q}_2$. 
\end{abstract}

\section{Introduction}

The goal of this paper is to prove the reciprocity conjecture of Khare and Wintenberger,
as found in \cite[Conjecture 5.5]{kw}.  The conjecture as originally stated asserts the
equality of two procyclic subgroups of the Galois group of the maximal abelian pro-$p$ extension
of a totally real field.  These subgroups fix the cyclotomic $\zp$-extension,
hence are conjecturally finite by Leopoldt's conjecture.  In fact, as we shall shortly explain, 
the reciprocity conjecture may be used to give an equivalent formulation of Leopoldt's conjecture 
in terms of ray class groups.  In this introduction, we give the original formulation and a slight extension of the conjecture and recall the motivation for the conjecture from \cite{kw}.   Let us begin straightaway with the formulation.

Let $p$ be an odd prime, let $F$ be a CM field containing $\mu_p$, and let $F^+$ be its maximal totally real subfield.  Let $F_{\infty}$ and $F_{\infty}^+$ denote the cyclotomic $\zp$-extensions of these respective fields, and let $\Gamma = \Gal(F_{\infty}/F)$.   Let $Q^+ = \{ \mf{q}_1, \mf{q}_2 \}$ be a set of two 
distinct finite primes of $F^+$ not lying over $p$, and suppose that these primes remain inert in 
$F_{\infty}^+$.  Let $\mc{M}$ (resp., $\mc{M}_{\infty}$) denote the maximal abelian pro-$p$ unramified outside $p$ extension of $F^+$ (resp., $F_{\infty}^+$).  

Let $\varphi_i$ denote the Frobenius
element in $\Gal(\mc{M}/F^+)$ attached to $\mf{q}_i$ for $i \in \{1,2\}$.  Each $\varphi_i$ restricts nontrivially to $F_{\infty}^+$, so the subgroup $\overline{\langle \varphi_1,\varphi_2 \rangle}$ 
topologically generated by $\varphi_1$ and $\varphi_2$ has a procyclic intersection with the kernel of this restriction map.  This provides the first of the two procyclic subgroups $M_Q$ and $N_Q$  of  $\Gal(\mc{M}/F_{\infty}^+)$ that are the subject of the reciprocity conjecture.  

\begin{definition} 
	The Frobenius line $M_Q$ is the
	intersection $\overline{\langle \varphi_1,\varphi_2 \rangle} \cap \Gal(\mc{M}/F_{\infty}^+)$.
\end{definition}

Let $A_{\infty}$ and $A_{\infty,\mf{q}}$ denote the $p$-parts of the
class group of $F_{\infty}$ and the ray class group of conductor $\mf{q}_1\mf{q}_2$
of $F_{\infty}$, respectively.  
We denote the minus parts of these groups, on which complex conjugation 
acts by $-1$, with a superscript ``$-$''.  
The definition of the ray class group gives rise to an exact sequence of $\zp[[\Gamma]]$-modules
\begin{equation} \label{simplerayseq}
	0 \to \mup \to A_{\infty,\mf{q}}^- \to A_{\infty}^- \to 0,
\end{equation}
the subgroup $\mup$ being identified with the $p$-power torsion in the residue field of 
$F_{\infty}^+$ for $\mf{q}_1$.  (Here, that the $\mf{q}_i$ are inert allows
the sequence to have this simple form, given a choice between the two primes.)  As shown by Iwasawa \cite[p.\ 75]{iwa-mod}, Kummer theory provides a canonical isomorphism between $\Hom(A_{\infty}^-,\mup)$ and
$\Gal(\mc{M}_{\infty}/F_{\infty}^+)$.  The $\Gamma$-coinvariants of the
latter group are canonically isomorphic to $\Gal(\mc{M}/F_{\infty}^+)$ and noncanonically,
i.e., up to a choice of generator of $\Gamma$, isomorphic to the continuous cohomology group 
$H^1(\Gamma,\Gal(\mc{M}_{\infty}/F_{\infty}^+))$.  
The next object was not named in \cite{kw}, so we give it a name to put in on a equal footing
with the Frobenius line. 
	
\begin{definition} 
	The ray class line $N_Q$ is the topological subgroup of $\Gal(\mc{M}/F_{\infty}^+)$ 
	generated by the image of the class of the extension \eqref{simplerayseq}
	of $\zp[[\Gamma]]$-modules in
	\begin{equation} \label{noncanon}
		H^1(\Gamma,\Gal(\mc{M}_{\infty}/F_{\infty}^+)) \cong \Gal(\mc{M}/F_{\infty}^+).
	\end{equation}
\end{definition}

We remark that while the class of \eqref{simplerayseq} depends up to sign on a choice of $\mf{q}_1$ over $\mf{q}_2$ and a choice of generator of $\Gamma$, the ray class line $N_Q$ does not.   We may now state the reciprocity conjecture of Khare and Wintenberger.  Despite its name, we will not employ class field theory in our approach to it.

\begin{conjecture}[Khare-Wintenberger]  
	The Frobenius line $M_Q$ and the ray class line $N_Q$ are equal.
\end{conjecture}

Let us briefly consider the connection with Leopoldt's conjecture for totally real fields \cite{leopoldt},  
which provides something of an impetus for the investigation of the relationship between $M_Q$ and $N_Q$.  We recall the statement for an arbitrary number field.

\begin{conjecture}[Leopoldt]
	Let $E$ be a number field.  
	Then the topological closure of the embedding of the
	units $\mc{O}_E^{\times}$ of $E$ in the direct product of the completions $E_v^{\times}$ at
	the primes $v$ of $E$ lying over $p$ has $\zp$-rank equal to the rank of $\mc{O}_E^{\times}$
	as an abelian group.
\end{conjecture}

The rank of $\mc{O}_E^{\times}$ in Leopoldt's conjecture is always at least
the $\zp$-rank of the topological closure, but Leopoldt's conjecture asserts that it is never
less.  As Iwasawa pointed out in \cite[Section 2.3]{iwa-ext}, Leopoldt's conjecture 
for a totally real number field is equivalent
to the statement that its maximal abelian, unramified outside $p$, pro-$p$ extension is a finite extension of its cyclotomic $\zp$-extension.

It is, of course, sufficient to prove Leopoldt's conjecture for totally real fields for fields $F^+$ of the sort we consider, since any totally real field sits inside of such an $F^+$.  And, as we have
just noted, Leopoldt's conjecture for $F^+$ is equivalent to the statement that $\Gal(\mc{M}/F_{\infty}^+)$ is finite.  The following result of Khare and Wintenberger  \cite[Corollary 6.5]{kw}, which is a consequence
of the \v{C}ebotarev density theorem, tells us that Leopoldt's conjecture for $F^+$ and $p$  is equivalent to the finiteness of the ray class lines $N_Q$ for all pairs $(\mf{q}_1,\mf{q}_2)$.

\begin{theorem}[Khare-Wintenberger] \label{equivform}
	Leopoldt's conjecture holds for $F^+$ and $p$ if and only if the Frobenius line $M_Q$
	is a finite group for every pair $(\mf{q}_1,\mf{q}_2)$ of primes of $F^+$ not
	lying over $p$ and inert in $F^+_{\infty}/F^+$.
\end{theorem}

In their work, Khare and Wintenberger explain that the reciprocity conjecture is a natural analogue in Iwasawa theory of a statement one can make for the generalized Jacobian $J_{P_1,P_2}$ of the singular curve obtained by identifying two rational points $P_1, P_2$ of a smooth projective curve $X$ over a field $k$ \cite[Section 5.3]{kw}.  
That is, the Tate module $T_{\ell}(J_{P_1,P_2})$ of the generalized Jacobian for a prime 
$\ell \neq \mr{char}\, k$
fits in an exact sequence of $\Z_{\ell}[[G_k]]$-modules
$$
	0 \to \Z_{\ell}(1) \to T_{\ell}(J_{P_1,P_2}) \to T_{\ell}(J) \to 0,
$$
where $J$ denotes the Jacobian of $X$ and $G_k$ is the absolute Galois group of $k$,
and where the map from $\Z_{\ell}(1)$ to the Tate module is given by fixing a choice of $P_1$ among the two primes.  
Using Weil duality, the class of this extension may be viewed as lying in the continuous
cohomology group $H^1(G_k,T_{\ell}(J))$, and this class is identified via Kummer theory 
with the class of the divisor $(P_2)-(P_1)$ in the $\ell$-part of $J(k)$.

We remark that the latter statement is a bit stronger than what would be analogous to 
the reciprocity conjecture,
which asserts an equality of subgroups, rather than elements.  However, this may be fixed.
We make several choices corresponding to a choice of $\mf{q}_1$ over $\mf{q}_2$.
Let $m_Q$ denote the unique element of $M_Q$ that is $\varphi_1 \varphi_2^j$ for some
$j \in \zp^{\times}$.  The $\mup$ found in \eqref{simplerayseq} is more canonically
written as $(\mup \oplus \mup)/\mup$, with the $i$th coordinate corresponding to the $p$-power
roots of unity in the residue field of $F_{\infty}^+$ at $\mf{q}_i$
for $i \in \{1,2\}$.  We consider the isomorphism taking the image 
of $(\alpha,\beta)$ to $\alpha\beta^{-1}$.
Finally, evaluation at a choice of generator of $\Gamma$ followed by restriction provides 
the isomorphism \eqref{noncanon}, and we may take the restriction of $\varphi_1$ as a canonical choice.  With these consistent choices, the class of \eqref{simplerayseq} defines a
generator $n_Q$ of $N_Q$.  A refined form of the reciprocity conjecture
may then be stated.  As we shall prove it, we state it as a theorem.

\begin{theorem} \label{recconjelts}
	The generators $m_Q$ and $n_Q$ of the Frobenius line and the ray class line,
	respectively, are the negatives of each other.
\end{theorem}

In fact, we will prove in Theorem \ref{reciprocity} 
a strengthening of this conjecture in which we allow $Q^+$ to be an arbitrary
finite set of finite primes of $F^+$ not lying over $p$ that are no longer assumed to
be inert in the extension $F_{\infty}^+/F^+$.

The structure of this paper
is simple.   In Section \ref{conjecture}, we reformulate and generalize the statement of the 
conjecture.  In Section \ref{proof}, we provide the proof of our more general statement.  Our 
proof will use little beyond standard facts regarding Galois cohomology and Kummer theory 
for number fields, hence the paucity of references.

\begin{ack}
	The author thanks Chandreshekhar Khare for introducing him to the reciprocity 
	conjecture and for his helpful comments on drafts of this work.  He was supported in 
	part during the preparation of this work by NSF award DMS-0901526. 
\end{ack}	

\section{The conjecture} \label{conjecture}

In this section, we start afresh and introduce the strengthening of the reciprocity conjecture that we intend to prove.  This requires a reasonable amount of notation and set-up, so let us embark upon the task.

Let $p$ be an odd prime number.  Fix a CM number field $F$ containing $\mu_p$, and let $F^+$ be its maximal totally real subfield.  Let $Q^+$ be a fixed finite set of finite primes of $F^+$ not lying over $p$ 
and $Q$ the set of primes of $F$ lying above them.  Let $\mf{q}$ denote the product of the
primes in $Q$.  More generally, for any algebraic extension $E$ of $F^+$,
we will let $Q_E$ denote the set of primes of $E$ lying over those in $Q^+$.  If $E$ is 
a number field, we let $\mf{q}_E$ denote the product of the primes in $Q_E$, and we let $E_w$ denote the completion of $E$ at $w \in Q_E$.

We let $F_{\infty}$ denote the cyclotomic $\zp$-extension of $F$ and $F_{\infty}^+$ its maximal totally real subfield.  The (finite) sets of primes above $Q^+$ in $F_{\infty}^+$ and $F_{\infty}$ will be
denoted $Q_{\infty}^+$ and $Q_{\infty}$, respectively.  Let $\mf{q}_{\infty}$ denote the
product of the primes in $Q_{\infty}$.
Set $\Gamma = \Gal(F_{\infty}/F)$ and $\Lambda = \zp[[\Gamma]]$.  

We let $\Omega$ denote the maximal unramified outside the primes over $p$ (or, $p$-ramified) extension
of $F$, and we set $G = \Gal(\Omega/F)$ and $H = \Gal(\Omega/F_{\infty})$.  As usual, for a $\Z[\Gal(F_{\infty}/F_{\infty}^+)]$-module $B$, we let 
$$	
	B^{\pm} = \{ b \in B \mid s(b) = \pm b \}
$$
for the generator $s$ of $\Gal(F_{\infty}/F_{\infty}^+)$, which is to say complex conjugation. 

Let $\mf{X}$ denote the Galois group of the maximal $p$-ramified abelian
pro-$p$ extension of $F$. 
Let $\mc{O}$ denote the ring of integers of  $F$.
Let $\Cl$ denote the class group of $\mc{O}$ and $A$ its $p$-part.
Similarly, let $\Cl_{\mf{q}}$ denote the ray class group of conductor $\mf{q}$ 
and $A_{\mf{q}}$ its $p$-part.  
The corresponding objects over $F_{\infty}$ will be denoted with a subscript $\infty$.  Note that $\mf{X}_{\infty}$ is a finitely generated (compact) $\La$-module, 
while $A_{\infty}$ and $A_{\infty,\mf{q}}$ are cofinitely generated (discrete) $\La$-modules.

For each $v \in Q^+$, let $D_v$ denote the
decomposition group at $v$ in $\mf{X}^+$, and let $\Gamma_v$ denote the decomposition
group at $v$ in $\Gamma$.  Note that $\Gamma_v$ is nontrivial as no prime ideal splits completely in the cyclotomic $\zp$-extension of a number field.  As $D_v$ is procyclic, this forces the restriction map $D_v \to \Gamma_v$ to be an isomorphism.
The direct sum of the inverses $\Gamma_v \to D_v$ of the restriction isomorphisms composed
with the inclusion maps $D_v \to \mf{X}^+$ yields a homomorphism
$$
	\delta \colon \bigoplus_{v \in Q^+} \Gamma_v \to \mf{X}^+.
$$
Note that the group $(\mf{X}_{\infty}^+)_{\Gamma}$ of $\Gamma$-coinvariants
is canonically isomorphic to the Galois group of the largest subextension of the maximal $p$-ramified
abelian pro-$p$ extension of $F_{\infty}^+$ that is actually abelian over $F^+$.
We therefore have a short exact sequence of abelian Galois groups
$$
	0 \to (\mf{X}_{\infty}^+)_{\Gamma} \to \mf{X}^+ \to \Gamma \to 0.
$$
We define $M_Q$ to be the image of the induced map $\delta^0$ in the commutative diagram
\begin{equation} \label{frobdiag}
	\SelectTips{cm}{} \xymatrix@C=20pt@R=14pt{
		0 \ar[d] & 0 \ar[d] \\
		\bigoplus_{Q^+}^0 \Gamma_v \ar@{-->}[r]^-{\delta^0} \ar[dd]
		& (\mf{X}_{\infty}^+)_{\Gamma} \ar[dd] \\ \\
		\bigoplus_{Q^+} \Gamma_v \ar[r]^-{\delta} \ar[dd] & \mf{X}^+ \ar[dd]\\ \\
		\Gamma \ar@{=}[r] & \Gamma \ar[d] \\
		& 0
	}
\end{equation}
with exact columns, 
where $\bigoplus^0$ denotes elements of the direct sum with trivial sum (or product), in
this case in $\Gamma$.  (As we have not assumed that any $v \in Q^+$ is inert in $F_{\infty}^+$, the left column need not be right exact.)
This image is the Frobenius module $M_Q$ of Khare-Wintenberger \cite[Definition 4.1]{kw}.    

By its definition, the $\mf{q}$-ray class group $\Cl_{\infty,\mf{q}}$ of $F_{\infty}$ fits into an 
exact sequence
$$
	1 \to \mc{O}_{\infty,\mf{q}}^{\times} \to  \mc{O}_{\infty}^{\times}
	\to (\mc{O}_{\infty}/\mf{q})^{\times} \to \Cl_{\infty,\mf{q}} \to \Cl_{\infty} \to 0,
$$
where $\mc{O}_{\infty,\mf{q}}^{\times}$ denotes the group of units in $\mc{O}_{\infty}$
that are congruent to $1$ modulo $\mf{q}$.
Since the minus part of  
$\mc{O}_{\infty}^{\times}$ is its group of roots of unity, this gives rise to
an exact sequence
\begin{equation} \label{rayclassseq}
	0 \to \mup(F_{\infty}) \to \bigoplus_{u \in Q_{\infty}^+} \mup(F^+_{\infty,u}) \to A_{\infty,\mf{q}}^-
	\to A_{\infty}^- \to 0
\end{equation}
on $p$-parts of minus parts, where the first map is the diagonal embedding.

Note that $A_{\infty}^-$ is Kummer dual to $\mf{X}_{\infty}^+$, which is
to say dual via homomorphisms to $\mup$.  Therefore, the 
Kummer dual of the exact sequence
\begin{equation} \label{shortrayseq}
	0 \to \bigoplus_{u \in Q_{\infty}^+}\!\!{}'\, \mup \to A_{\infty,\mf{q}}^- \to A_{\infty}^- 
	\to 0
\end{equation}
induced by \eqref{rayclassseq}, where $\bigoplus'$ denotes quotient by the diagonal,
has the form
\begin{equation} \label{dualrayseq}
	0 \to \mf{X}_{\infty}^+ \to \Hom(A_{\infty,\mf{q}}^-,\mup) \to 
	\bigoplus_{u \in Q_{\infty}^+}\!\!{}^0\, \zp \to 0.
\end{equation}
In $\Gamma$-homology, we have canonical isomorphisms $H_1(\Gamma,\zp) \cong \Gamma$ and
$$
	H_1\Bigg(\Gamma,\bigoplus_{\substack{u \in Q_{\infty}^+ \\ u \mid v}} \zp\Bigg)
	\cong \Bigg(\bigoplus_{\substack{u \in Q_{\infty}^+ \\ u \mid v}} \Gamma
	 \Bigg)\Big.^{\Gamma}
	\cong \Gamma_v
$$
for each $v \in Q^+$, the last map being induced by multiplication in $\Gamma$.
It follows that the connecting homomorphism in the $\Gamma$-homology of \eqref{dualrayseq} has the form
$$
	\kappa^0 \colon \bigoplus_{v \in Q^+}\!\!{}^0\, \Gamma_v \to (\mf{X}_{\infty}^+)_{\Gamma}.
$$

With both $\delta^0$ and $\kappa^0$ are canonically defined, it is reasonable to consider the following
statement.  

\begin{theorem} \label{reciprocity}
	The map $\kappa^0$ is equal to $\delta^0$.
\end{theorem}

Before we proceed to the proof, let us compare Theorem \ref{reciprocity} with the formulation of the reciprocity conjecture given in \cite{kw} and the introduction.  So, suppose that $|Q^+_{\infty}| = 2$.  In this case, the image of $\kappa^0$ is the ray class line $N_Q$.  More precisely, Khare and Wintenberger consider the class 
$\xi_Q$ of the exact sequence \eqref{shortrayseq} in 
\begin{multline*}
	H^1(\Gamma,\Hom(A_{\infty}^-,(\mup \oplus \mup)')) \cong
	H^1(\Gamma,\Hom(A_{\infty}^-,\mup))\\ \cong H^1(\Gamma,\mf{X}_{\infty}^+)
	\cong \Hom(\Gamma,(\mf{X}_{\infty}^+)_{\Gamma}),
\end{multline*}
where the isomorphism $(\mup \oplus \mup)' \xrightarrow{\sim} \mup$ takes $(\alpha,\beta)$ to $\alpha \beta^{-1}$, which requires a choice of one of the two primes in $Q$.  They then evaluate $\xi_Q$
on a noncanonical choice of generator of $\Gamma$  and define $N_Q$ to be the $\zp$-submodule it generates  \cite[Section 5.1]{kw}, which is independent of the choices made.  We claim that $\xi_Q$ is the negative of the class in $\Gamma$-cohomology of the dual sequence \eqref{dualrayseq}, 
which also lies in
$$
	H^1(\Gamma,\Hom((\zp \oplus \zp)^0,\mf{X}_{\infty}^+)) \cong
	H^1(\Gamma,\mf{X}_{\infty}^+),
$$
the isomorphism $(\zp \oplus \zp)^0 \xrightarrow{\sim} \zp$ used here taking $(a,-a)$ to $a$, with
the first coordinate corresponding to the same prime of $Q$ as before.  The latter class
as viewed in $(\mf{X}_{\infty}^+)_{\Gamma}$ agrees with $\kappa^0$ evaluated on the chosen generator.  Supposing the claim, since the $\zp$-span of the image of $\delta^0$ is by definition the Frobenius line $M_Q$, we see that Theorem \ref{reciprocity} implies the reciprocity conjecture.  Moreover, in the notation of the introduction, we have $m_Q = \delta^0(\varphi_1)$ and $n_Q = -\kappa^0(\varphi_1)$, so
Theorem \ref{reciprocity} implies Theorem \ref{recconjelts} as well.

As for the claim, it follows easily from the following simple lemma.

\begin{lemma} \label{pdnegclass}
	The class in $H^1(\Gamma,\Hom(C,A))$ of an exact sequence of finitely generated 
	$\Lambda$-modules
	$$
		0 \to A \xrightarrow{\iota} B \xrightarrow{\pi} C \to 0
	$$
	is sent to the negative of the class in $H^1(\Gamma,\Hom(A^{\vee},C^{\vee}))$ of the
	Pontryagin dual exact sequence under the isomorphism
	$$
		\Hom(C,A) \xrightarrow{\sim} \Hom(A^{\vee},C^{\vee}), \qquad f \mapsto f^{\vee}
	$$
	given by $f^{\vee}(\phi) = \phi \circ f$ for $f \in \Hom(C,A)$ and $\phi \in A^{\vee}$.
\end{lemma}

\begin{proof}
	For any set-theoretic map $\kappa \colon M \to N$
	of $\Lambda$-modules and $\gamma \in \Gamma$, we define $\kappa^{\gamma} \colon M \to N$ by
	 $\kappa^{\gamma}(m) = \gamma \kappa(\gamma^{-1}m)$ for each $m \in M$.
	For any continuous function $s$ splitting $\pi$, the class of our exact sequence is that of the
	cocycle $\chi$ given by $\iota \circ \chi(\gamma)= s^{\gamma} - s$.
	Let $t \colon B \to A$ be the unique map satisfying $\iota \circ t = 1 - s \circ \pi$.
	The class of the Pontryagin dual sequence is that of the cocycle $\chi^* \colon \Gamma
	\to \Hom(A^{\vee},C^{\vee})$ given by
	$$
		\pi^{\vee}(\chi^*(\gamma)(\phi)) = (t^{\vee})^{\gamma}(\phi) - t^{\vee}(\phi)
		= \phi \circ (t^{\gamma} -t)
	$$
	for $\phi \in A^{\vee}$.  Taking the dual of
	$\chi(\gamma)$ and noting that $\iota \circ t^{\gamma} = 1-s^{\gamma} \circ \pi$, we obtain
	$$
		\pi^{\vee}(\chi(\gamma)^{\vee}(\phi)) = \phi \circ \iota^{-1} \circ (s^{\gamma} - s) \circ \pi
		= \phi \circ (t-t^{\gamma}),
	$$
	as desired.
\end{proof}

\section{The proof} \label{proof}

To begin with, note that the class group of the maximal $p$-ramified extension $\Omega$ of $F$ is trivial as $\Omega$ contains the Hilbert class field of every number field in $\Omega$. 
Let us set
$$
	\mc{W} = \dirlim{E \subset \Omega} \bigoplus_{w \in Q_E} \mup(E_w),
$$
where the direct limit runs over the number fields $E$ containing $F$ in $\Omega$.
Note that $\mc{W}$ is $H$-induced from $\bigoplus_{Q_{\infty}} \mup$, so
$$
	\mc{W}^H = \bigoplus_{u \in Q_{\infty}} \mup
$$
and $H^i(H,\mc{W}) = 0$ for $i \ge 1$.  
For any $u \in Q^+_{\infty}$, the $u$-component of $\mc{W}^H$ will be $\mup$
if the prime $v$ below $u$ is inert in $F/F^+$, and it will be $\mup \oplus \mup$, with
complex conjugation
acting by switching coordinates, if the prime $v$ splits in $F/F^+$.  Either way, the minus 
part is $\mup$.  Therefore, we have
$$
	(\mc{W}^H)^- = \bigoplus_{u \in Q^+_{\infty}} \mup.
$$

Now set $\omega = \mc{W}/\mup$ with respect to the diagonal embedding of $\mup$.  We then obtain a long exact sequence in $H$-cohomology, the minus
part of which begins
\begin{equation} \label{seqcompray}
	1 \to \mup \to \bigoplus_{u \in Q^+_{\infty}} \mup \to (\omega^H)^-
	\to A_{\infty}^- \to 0,
\end{equation}
the term $A_{\infty}^-$ being $H^1(H,\mup)^-$ by Kummer theory.

We make the following claim.

\begin{proposition} \label{rayprop}
	There is an isomorphism $\phi \colon 
	A_{\infty,\mf{q}}^- \to (\omega^H)^-$
	of $\La$-modules fitting into an isomorphism of exact sequences of $\Lambda$-modules
	$$
		\SelectTips{cm}{}
		\xymatrix{
			1 \ar[r] & \mup \ar[r] \ar@{=}[d] & \bigoplus_{Q^+_{\infty}} \mup
			\ar[r] \ar@{=}[d] & A_{\infty,\mf{q}}^- \ar[r] \ar[d]^{\phi} & A_{\infty}^-
			\ar[r] \ar@{=}[d] & 0 \\
			1 \ar[r] & \mup \ar[r] & \bigoplus_{Q^+_{\infty}} \mup \ar[r] &
			(\omega^{H})^- \ar[r] & A_{\infty}^- \ar[r] & 0
		}
	$$
	from \eqref{rayclassseq} to \eqref{seqcompray}.
\end{proposition}

\begin{proof}
	We define $\phi$ explicitly.  Let $\mf{a}$ be a nonzero finitely generated
	fractional ideal of 
	$\mc{O}_{\infty}$ that is relatively prime to $p$ and the ideals in $Q_{\infty}$, and 
	suppose that its
	ray class $[\mf{a}]_{\mf{q}}$ lies in $A_{\infty,\mf{q}}^-$.  Let $[\mf{a}]$ denote 
	its class in $A_{\infty}^-$.  Let $n$ be such that $\mf{a}^{p^n}$ is principal, 
	therefore generated by some $a \in F_{\infty}^{\times}$.  We assume, as is possible, that
	$\mf{a}$ and $a$ are chosen so that $s\mf{a} = \mf{a}^{-1}$ and $s(a) = a^{-1}$,
	where $s$ denotes the generator of $\Gal(F_{\infty}/F_{\infty}^+)$.
	Let $\alpha$ be such that 
	$\alpha^{p^n} = a$, and note that $\alpha \in \Omega^{\times}$, since $\alpha$ 
	generates a $p$-ramified extension of $F_{\infty}$.  
	Since $[\mf{a}]_{\mf{q}}$ has $p$-power order, we have
	$a^{p^t} \equiv 1 \bmod \mf{q}_{\infty}$ for sufficiently large $t$.  Therefore, the image 
	$$
		\bar{\alpha} \in \dirlim{E \subset \Omega} (\mc{O}_E/\mf{q}_E)^{\times}
	$$
	of $\alpha$ has $p$-power order, so it lies in the $p$-power torsion subgroup $\mc{W}$
	of the direct limit. 
	For $\sigma \in H$, we have
	$\bar{\alpha}^{\sigma-1} = \alpha^{\sigma-1} \in \mu_{p^n}$, so the image 
	$\phi([\mf{a}]_{\mf{q}})$ of $\bar{\alpha}$ in $\omega$ lies in $\omega^{H}$.  
	
	Note that if $\tau \in \Gal(\Omega/F^+)$,
	then the class $\tau([\mf{a}]_{\mf{q}})$ is represented by $\tau\mf{a}$, the $p^n$th power
	of which is generated
	by $\tau(a)$, and this has $\tau(\alpha)$ as a $p^n$th root.  This, in turn, has image
	$\tau(\bar{\alpha})$ in $\mc{W}$.  In particular,
	since $[\mf{a}]_{\mf{q}}$ lies in the minus part of $A_{\infty,\mf{q}}$, the element
	$\phi([\mf{a}]_{\mf{q}})$ lies in the minus part of $\omega^{H}$.
	Since $\phi$ is multiplicative given consistent choices, 
	if we can show that $\phi([\mf{a}]_{\mf{q}})$ is well-defined, we will then have
	constructed a homomorphism $\phi$ of $\Lambda$-modules.
	
	We show that $\phi$ is well-defined.  First, suppose that $n$ is replaced by some $m \ge n$
	and $a$ by $a^{p^{m-n}}$.  Then $\alpha$ can still be taken to be the same element.
	On the other hand, for a fixed $n$, we may replace $\alpha$ by a product of it with
	a $p^n$th root of unity. However, this does not affect the image of $\bar{\alpha}$ in 
	$\omega$.  Next, any generator of $\mf{a}^{p^n}$ in $(F_{\infty}^{\times})^-$
	will differ from $a$ by some element of the minus part 
	$\mup$ of $\mc{O}_{\infty}^{\times}$, the $p^n$th root of this element will be another 
	$p$-power root
	of unity, and again this will not change the image of $\bar{\alpha}$ in $\omega$.  Finally, any
	finitely generated fractional ideal $\mf{b}$ prime to $Q_{\infty}$ and $p$ and
	representing the class $[\mf{a}]_{\mf{q}}$ has the form $\mf{b} = \mf{a}(c)$,
	where $c \in F_{\infty}^{\times}$ with $c \equiv 1 \bmod \mf{q}_{\infty}$.  
	Then $\mf{b}^{p^n}$ is generated by
	$ac^{p^n}$, which has $p^n$th root $\alpha c$.  The image of this element in the direct
	limit of the groups $(\mc{O}_E/\mf{q}_E)^{\times}$ is just $\bar{\alpha}$.  Having
	checked independence of all choices, we have that $\phi$ is well-defined.
	
	It remains only to check the commutativity of the diagram.  For this, suppose that
	$(\alpha_u)_u \in \bigoplus_{u \in Q_{\infty}^+} \mup$, and let $a \in \mc{O}_{\infty}$
	with $s(a) = a^{-1}$ be such that  
	$a$ has image $\alpha_u$ in the residue field of $\mc{O}_{\infty}$ 
	for each of the one or two primes of $Q_{\infty}$ lying over each $u \in Q_{\infty}^+$.  Then
	the class $[(a)]_{\mf{q}}$ is represented by $(a)$
	and has image $(\alpha_u)_u$ in $\bigoplus_{u \in Q_{\infty}^+} \mup$ and therefore
	the image of that in $\omega^H$.  Hence, the middle square commutes.  As for
	the right-hand square, let $[\mf{a}]_{\mf{q}} \in A_{\infty,\mf{q}}^-$, and note
	that $\phi([\mf{a}]_{\mf{q}})$ is the image of $\bar{\alpha}$ as before.  The connecting
	homorphism $\omega^{H} \to H^1(H, \mup)$ then takes $\phi([\mf{a}]_{\mf{q}})$
	to the cocycle $\tau \mapsto \bar{\alpha}^{\tau - 1}$ for $\tau \in H$, which as we
	have already mentioned
	is just the Kummer character $\tau \mapsto \alpha^{\tau-1}$.
	Since $\mf{a}^{p^n} = (a)$ and $\alpha^{p^n} = a$,
	this character has image $[\mf{a}]$ in $A_{\infty}^-$, as desired.
\end{proof}

Having shown that the class of 
\begin{equation} \label{shortcompseq}
	1 \to \bigoplus_{u \in Q_{\infty}^+}\!\!{}'\, \mup \to (\omega^H)^- \to A_{\infty}^- \to 0
\end{equation}
agrees with the class of \eqref{shortrayseq},
the first connecting homomorphism in the long exact sequence in the $\Gamma$-cohomology of \eqref{shortcompseq} will now be the negative of the Pontryagin dual of $\kappa^0$ (as follows from Lemma \ref{pdnegclass}).  In order to compare this with $\delta^0$, we will reconstruct the diagram \eqref{frobdiag} in a manner that incorporates the latter sequence.

Let us start by defining the needed object $\mc{V}$, which most simply put is just the twist 
$\mc{W}(-1)$.  More meticulously, for each $v \in Q$, we can set 
$$
	\mc{V}_v =  \dirlim{E \subset \Omega} \bigoplus_{\substack{w \in Q_E\\ w \mid v}}
	\qp/\zp
$$
and then take
$$
	\mc{V} = \bigoplus_{v \in Q} \mc{V}_v = 
	\dirlim{E \subset \Omega} \bigoplus_{w \in Q_E} \qp/\zp.
$$
Let $\nu$ be the quotient of $\mc{V}$ by $\qp/\zp$ embedded diagonally.  
As with $\mc{W}$, note that $\mc{V}$ is the $H$-induced module 
of $\bigoplus_{Q_{\infty}} \qp/\zp$ (in that $\mc{W}$ is the
Tate twist of $\mc{V}$).
Therefore, the long exact sequence
in $H$-cohomology attached to
\begin{equation} \label{nuquot}
	0 \to \qp/\zp \to \mc{V} \to \nu \to 0
\end{equation}
gives rise to a short exact sequence
\begin{equation} \label{neededseq}
	0 \to \bigoplus_{u \in Q_{\infty}}\!{}'\, \qp/\zp \to \nu^H
	\to \mf{X}_{\infty}^{\vee} \to 0.
\end{equation}

There is a commutative diagram 
\begin{equation} \label{bigdiag}
	\SelectTips{cm}{} \small
	\xymatrix@C=10pt@R=15pt{
		&&&  & 
		& H^1(\Gamma,\mc{V}^H ) \ar[r] \ar[d]^{\wr}_{\beta} 
		& H^1(\Gamma,\nu^{H}) \ar[r] \ar[d]^{\wr}_{\theta}
		& H^1(\Gamma,\mf{X}_{\infty}^{\vee}) \ar[d]^{\wr}_{\psi} \ar[r] & 0 \\
		0 \ar[r] & \qp/\zp \ar[r]  & \mc{V}^G \ar[r] 
		& \nu^{G} \ar[r] 
		& H^1(G,\qp/\zp) \ar[r] 
		& H^1(G, \mc{V}) \ar[r]
		& H^1(G,\nu)  \ar[r]
		& H^2(G,\qp/\zp) \ar[r] & 0 
	}
	\normalsize
\end{equation}
in which $\beta$ and $\theta$ are inflation maps, the
lower row is the long exact sequence in $G$-cohomology
arising from \eqref{nuquot}, and $\psi$ is the unique map making the diagram commute. 
The right exact top row is induced on $\Gamma$-cohomology by the right exact sequence
$$
	\mc{V}^H \to \nu^H \to \mf{X}_{\infty}^{\vee} \to 0
$$ 
giving rise to \eqref{neededseq}, noting that $\Gamma$ has cohomological dimension $1$.
The surjectivity of $\beta$ follows from the inflation-restriction sequence and the fact that
$H^1(H,\mc{V}) = 0$.   Similarly, the cokernel of $\theta$ injects into 
$H^1(H,\nu) \cong H^2(H,\qp/\zp)$, and the latter group is trivial by weak Leopoldt 
(\cite{iwa-ext}, but see \cite[Theorem 10.3.22]{nsw} for this form).
The interested reader may check that $\psi$ arises as the map
$E_2^{1,1} \to E^2$ (noting that $E_2^{i,j} = 0$ for $i \ge 2$)
in the Hochschild-Serre spectral sequence
 $$
	E_2^{i,j} = H^i(\Gamma,H^j(H,\qp/\zp)) \Rightarrow E^{i+j} = H^{i+j}(G,\qp/\zp),
$$
(i.e., there is a morphism from the corresponding spectral sequence with $\nu$ coefficients to this
sequence with a shift of degree).

Replacing the last three terms in the bottom row of \eqref{bigdiag} using the isomorphisms with
the terms of the top row and mapping the resulting exact sequence to 
the long exact sequence in $\Gamma$-cohomology attached to \eqref{neededseq}, 
we obtain a diagram
\begin{equation} \label{mapseq}
	\small
	\SelectTips{cm}{}
	\xymatrix@C=9.5pt{
	0 \ar[r] & \bigoplus_Q' \qp/\zp \ar[r] \ar@{^{(}->}[d] & \nu^{G} \ar[r] \ar@{=}[d]
	& \mf{X}^{\vee} \ar[r]^-{\iota} \ar@{->>}[d]^{\Res} 
	& H^1(\Gamma,\bigoplus_{Q_{\infty}} \qp/\zp)\ar[r] \ar@{->>}[d]^{-\pi}
	& H^1(\Gamma,\nu^H) \ar[r] \ar@{=}[d] 
	& H^1(\Gamma,\mf{X}_{\infty}^{\vee}) \ar[r] \ar@{=}[d] & 0 \\
	0 \ar[r] & (\bigoplus_{Q_{\infty}}' \qp/\zp)^{\Gamma} \ar[r] & \nu^{G} \ar[r] 
	& (\mf{X}_{\infty}^{\vee})^{\Gamma}  \ar[r]^-{\partial} 
	& H^1(\Gamma,\bigoplus_{Q_{\infty}}' \qp/\zp) \ar[r]
	& H^1(\Gamma,\nu^H) \ar[r]
	& H^1(\Gamma,\mf{X}_{\infty}^{\vee}) \ar[r] & 0.
	}
	\normalsize
\end{equation}
Here, the map $\Res$ is restriction on homomorphism groups, 
and the map $\pi$ is induced by the quotient map
$$
	\bigoplus_{u \in Q_{\infty}} \qp/\zp \to \bigoplus_{u \in Q_{\infty}}\!\!{}'\, \qp/\zp.
$$  
Clearly, $\Res$ has kernel $\Gamma^{\vee}$ and $\pi$ has kernel an infinite quotient of
$\Gamma^{\vee}$.
However, it is not immediately clear that the middle square in the 
diagram commutes, so we prove this.

\begin{proposition} \label{mapdiagcomm}
	The diagram \eqref{mapseq} is commutative.
\end{proposition}

\begin{proof}
	The only square for which commutativity is nonobvious is the third, so we focus on it.
	Let $\chi \in G^{\vee} = \mf{X}^{\vee}$, and
	restrict $\chi$ to $H$ (so, to an element of $(\mf{X}_{\infty}^{\vee})^{\Gamma}$). 
	Find $a \in \mc{V}$ such that the image $\bar{a}$ of $a$ in $\nu$ is fixed by $H$
	and maps to $\chi|_H$ under the natural map
	$\nu^H \to \mf{X}_{\infty}^{\vee}$. That is, $\chi(\tau) =  (\tau-1)a$ for all $\tau \in 
	H$.  Then $\partial(\chi|_H)$ is the unique homomorphism such that
	$$
		\partial(\chi|_H)(\gamma) = (\tilde{\gamma}-1)\bar{a}
	$$
	for every $\gamma \in \Gamma$ and any lift of it to $\tilde{\gamma} \in G$.
	This homomorphism takes values in $\bigoplus_{Q_{\infty}}' \qp/\zp$ inside
	$\nu^H$, since $\chi|_H$ is fixed by $\Gamma$.
	
	On the other hand, we may view $\chi$ as an element of $H^1(G,\mc{V})$ via
	the diagonal embedding of $\qp/\zp$ in $\mc{V}$.  Since $\chi(\tau) = (\tau-1)a$
	for $\tau \in H$, its image 
	$\iota(\chi)$ in $H^1(\Gamma,\bigoplus_{Q_{\infty}} \qp/\zp)$ under the inverse of
	inflation is represented by a cocycle given by
	\begin{equation} \label{shapiro}
		\iota(\chi)(\gamma) = \chi(\tilde{\gamma}) - (\tilde{\gamma}-1)a
	\end{equation}
	for $\gamma \in \Gamma$, viewing $\bigoplus_{Q_{\infty}} \qp/\zp$ as $\mc{V}^H$.
	We then have
	$$
		\pi(\iota(\chi))(\gamma) = -(\tilde{\gamma}-1)\bar{a}
	$$ 
	since 
	$\chi(\gamma) \in \qp/\zp$, so $\pi \circ \iota = -\partial \circ \Res$, as claimed.
\end{proof}

We now define 
$$
	\kappa^0 \colon \bigoplus_{v \in Q^+}\!\!{}^0\, \Gamma_v \to (\mf{X}_{\infty}^+)_{\Gamma}
$$ 
to be the negative of the plus part of the Pontryagin dual of the connecting homomorphism
$\partial$
and
$$
	\kappa \colon \bigoplus_{v \in Q^+} \Gamma_v \to \mf{X}^+
$$
to be the plus part of the Pontryagin dual of the map $\iota$ found in \eqref{mapseq}.
This definition of $\kappa^0$ agrees with that Section \ref{conjecture}.
The plus part of the Pontryagin dual of the third commutative square in \eqref{mapseq} 
then yields the commutative square
\begin{equation} \label{commsq}
	\SelectTips{cm}{} \xymatrix@C=20pt@R=32pt{
		\bigoplus_{v \in Q^+}^0 \Gamma_v \ar[r]^-{\kappa^0} \ar[d] & 
		(\mf{X}_{\infty}^+)_{\Gamma} \ar[d] \\
		\bigoplus_{v \in Q^+} \Gamma_v \ar[r]^-{\kappa} & \mf{X}^+,
	}
\end{equation}
in which the right and left vertical maps are the canonical injections with cokernel $\Gamma$
and contained in $\Gamma$, respectively.  
In the proof of the following, which implies Theorem \ref{reciprocity}, we shall 
see that $\kappa = \delta$ and therefore $\kappa^0 = \delta^0$, so the diagram 
\eqref{commsq} agrees with \eqref{frobdiag}.

We may now prove our main theorem.

\begin{proof}[Proof of Theorem \ref{reciprocity}]
	By the commutativity of \eqref{commsq}, we need only
	show that $\kappa$ is the map $\delta$ of the introduction.  Note that
	we have already described a map $\iota$ in the proof of
	Proposition \ref{mapdiagcomm}, the plus part of which may be identified with
	the Pontryagin dual of $\kappa$.  So, we consider the plus part
	$$
		\iota^+ \colon (\mf{X}^+)^{\vee} \to 
		H^1\Big(\Gamma,\bigoplus_{u \in Q_{\infty}^+} \qp/\zp\Big)
	$$  
	of $\iota$.
	The map $\iota^+$ may be viewed as a collection of 
	maps $\iota^+_v \colon (\mf{X}^+)^{\vee} \to \Gamma_v^{\vee}$ for each $v \in Q^+$
	via the isomorphism
	$$
		H^1\Bigg(\Gamma,\bigoplus_{\substack{u \in Q_{\infty}^+ \\ u \mid v}} \qp/\zp\Bigg)
		\xrightarrow{\sim} \Hom(\Gamma_v,\qp/\zp)
	$$	
	that takes a cocycle $f$ to its restriction to $\Gamma_v$, which has values in
	the diagonal $\qp/\zp$ in the direct sum.
	
	By \eqref{shapiro}, for $\chi \in (\mf{X}^+)^{\vee}$, which we may think of as
	an element of $G^{\vee}$, and $\gamma_v \in \Gamma_v$, 
	we may write
	$$
		\iota^+_v(\chi)(\gamma_v) = \chi(\tilde{\gamma}_v) - (\tilde{\gamma}_v-1)a_v.
	$$
	for some lift $\tilde{\gamma}_v$ of $\gamma_v$ in $G$ and $a_v \in \mc{V}_v$.  
	By definition of $\mc{V}_v$, there exists a number field $E$ containing $F$ such
	that $a_v$ lies in the direct sum $\bigoplus_{w \mid v} \qp/\zp$ over primes $w$ of $E$
	lying over $v$, and $\tilde{\gamma}_v$ acts on the sum by permuting the coordinates
	of its elements.  In particular, the $w$-coordinate of $\iota^+_v(\chi)(\gamma_v)$ in this direct sum is
	$$
		\chi(\tilde{\gamma}_v) - (a_v)_{\tilde{\gamma}_v^{-1}(w)} + (a_v)_w.
	$$
	If we pick $\tilde{\gamma}_v$
	to lie in the decomposition group of $G$ at a place over $w$, then the $w$-coordinate of
	$\iota^+_v(\chi)(\gamma_v)$ will equal 
	$\chi(\tilde{\gamma}_v)$, but then every
	$w$-coordinate for $w$ lying over  $v$ must have this value since $\iota^+_v(\chi)$ takes
	values on $\Gamma_v$ in the diagonally embedded $\qp/\zp$ in $\mc{V}_v$.  
	As $\tilde{\gamma}_v$ will restrict to an element
	of the decomposition group $D_v$ in $\mf{X}^+$, the $v$-component $\delta_v$ of $\delta$
	satisfies
	$$
		\delta_v^{\vee}(\chi)(\gamma_v) = \chi(\tilde{\gamma}_v) = 
		\iota_v^+(\chi)(\gamma_v).
	$$
	For the sums of the Pontryagin duals, this says exactly that $\kappa = \delta$.
\end{proof}
	
\renewcommand{\baselinestretch}{1}

\end{document}